\documentclass[A4paper,12pt]{article}
\usepackage{latexsym}
\usepackage{mathrsfs}
\usepackage{amssymb}
\usepackage{amscd}
\usepackage[dvips]{graphicx}                  

\newtheorem{theorem}{Theorem}

\newenvironment{proof}[1][Proof]{\textbf{#1.} }{\ \rule{0.5em}{0.5em}}
\long\def\symbolfootnote[#1]#2{\begingroup%
	\def\thefootnote{$\;$}\footnote[#1]{$^*$#2}\endgroup}
\begin{document}
	
	\title{Partitions of non-complete Baire metric spaces}
	\author{Ryszard Frankiewicz and Joanna Jureczko\footnote{The author is partially supported by Wroc\l{}aw Univercity of Science and Technology grant no. 049U/0032/19.}}
\maketitle

\symbolfootnote[2]{Mathematics Subject Classification: Primary 03C25, 03E35, 03E55, 54E52.

	\hspace{0.2cm}
	}

\begin{abstract}
We investigate the properties of ideals associated with Kuratowski partitions of non-complete Baire metric spaces. We show that such an ideal can be precipitous.
\end{abstract}

\maketitle

\section{Introduction}

In 1935 K. Kuratowski in \cite{KK} posed the problem  whether a function $f \colon X \to~Y$, (where $X$ is completely metrizable and $Y$ is metrizable), such that each preimage of an open set of $Y$ has the Baire property, is continuous apart from a meager set.

 R. H. Solovay and L. Bukovsk\'y independently proved the non-existence of Kuratowski partitions of the unit interval $[0,1]$ for measure and category by forcing methods (and the generic ultrapower), but Bukovsk\'y's proof, (see \cite{LB}), is shorter and less complicated than Solovay's (unpublished results).

In \cite{EFK} there is shown that this problem is equivalent to the problem of the existence of partitions of completely metrizable spaces into meager sets with the property that the union of each subfamily of this partition has the Baire property. Such a partition is called  \textit{Kuratowski partition}, (see the next section for a formal definition).

In paper \cite{JJ} there was introduced the notion of $K$-ideals associated with Kuratowski partitions and there were examined their properties.
It can be supposed that from a structure of such  $K$-ideal one can "decode" complete information about Kuratowski partition of a given space. Unfortunately, this is not the case because, as shown in \cite{JJ}, the structure of such an ideal can be almost arbitrary, i.e. it can be the  Fr\'echet ideal, so by
\cite[Lemma 22.20, p. 425]{TJ} it is not precipitous, whenever $\kappa$ is regular.
Moreover, as demonstrated in \cite{JJ},  for measurable cardinal $\kappa$, a $\kappa$-complete ideal can be represented by some $K$-ideal.
However if $\kappa =|\mathcal{F}|$ is not measurable cardinal, where $\mathcal{F}$ is Kuratowski partition of a given space, then one can obtain an $|\mathcal{F}|$-complete ideal which can be the Fr\'echet ideal or a $\kappa$-complete ideal representating some $K$-ideal or can  be a proper ideal of such $K$-ideal and contains the Fr\'echet ideal. 
Thus, for obtaining Kuratowski partition from  $K$-ideal we need to have complete information about the space in which the ideal is considered.

In the presence of above considerations and the statement that 
 ZFC + "there exists precipitous ideal" is equiconsistent with ZFC + "there exists measurable cardinal", (see \cite{FK}), the natural question is: under which assumptions $K$-ideals can be precipitous, see  \cite[Theorem 22.33, p.432]{TJ}. Such information can lead us to prove the required implication mentioned above. 
 Our work in this topic (divided into two papers) enlarges results of \cite{FK}, where the authors  proved among others   that ZFC + "there exists a  Kuratowski partition of a Baire metric space" is consistent, then ZFC + "there exists  measurable cardinal" is consistent as well, by using  forcing methods  (i. e. a model of the G-generic ultrapower in Keisler sense, see \cite[sec. 6.4]{CK} and \cite{TJ} for details). 
 In this paper we show that $K$-ideal of non-complete Baire metric space can be precipitous.  In \cite{FK} there is shown combinatorial proof of the following statement: if a cardinal $\kappa$ is measurable, then there exists  a complete metric Baire space with  Kuratowski partition of size $\kappa$. Thus, it remains to show the converse implication. (As it will be shown in the proof, the assumption of completeness of this space cannot be omitted). This result will be given in a separate paper.

This paper consists of three  sections. Section 2 contains definitions and previous results concerning among others a Kuratowski partition,  and a precipitous ideal. 
Section 3 there is presented the main result: if $X$ is a  Baire metric space with  Kuratowski partition  $\mathcal{F}$ of size $\kappa$, ($\kappa$ is regular and uncountable) and $I_{\mathcal{F}}$ is a $K$-ideal associated with $\mathcal{F}$, then there exists an open set $U \subseteq X$ such that the $K$-ideal $I_{\mathcal{F}\cap U}$ is a precipitous ideal on $\kappa$.  
Since the proof of this result seems to be a bit complicated, we decided to present the  second proof of this theorem in the game theoretic notion. Namely,  we will use the game theoretic characterisation of precipitous ideals, see \cite{JG}. In this place it is worth emphasizing that the notion "$\alpha$-favorable", used in Section 3, comes from Choquet, (see \cite{C}, where the reader can also find more information about equivalences of Baire spaces in terms of games). The paper is finished with Section 4 including open problem concerning possibilities of enlarging our results for weakly- and pseudo precipitous ideals.

For definitions and  facts not cited here we refer to e.g. \cite{RE, KK1} (topology) and \cite{TJ} (set theory).

\section{Definitions and previous results}

\textbf{2.1.} Let $X$ be a topological space.  A set $U \subseteq X$ has \textit{the Baire property} iff there exist an open set $V \subset X$ and a meager set $M \subset X$ such that $U = V \triangle M$, where $\triangle$ means the symmetric difference of sets.
\\
\\
\textbf{2.2.} A partition $\mathcal{F}$ of $X$ into meager subsets of $X$ is called \textit{Kuratowski partition} iff $\bigcup \mathcal{F}'$ has the Baire property for all $\mathcal{F}' \subseteq \mathcal{F}$.
If there exists a Kuratowski partition of $X$ we always denote by $\mathcal{F}$ with the smallest cardinality $\kappa$. Moreover, we enumerate $$\mathcal{F} = \{F_\alpha \colon \alpha < \kappa\}.$$
Obviously,  $\kappa$ is regular. If $\kappa$ was singular, then $cf(\kappa)$ would be the minimal one. By Baire Theorem $\kappa$ is uncountable.

For a given set $U \subseteq X$  the family 
$$\mathcal{F}\cap U = \{F \cap U \colon F \in \mathcal{F}\}$$
is Kuratowski partition of $U$ as a subspace of $X$.
\\
\\
\textbf{2.3.} With any Kuratowski partition 
$\mathcal{F} = \{F_\alpha \colon \alpha < \kappa\}$, indexed by a cardinal $\kappa$,  one may associate an ideal 
$$I_\mathcal{F} = \{A \subset \kappa \colon \bigcup_{\alpha \in A} F_\alpha \textrm{ is meager}\}$$
which is called \textit{$K$-ideal}, (see \cite{JJ}).
\\
Note, that $I_\mathcal{F}$ is a non-principal ideal. Moreover, $[\kappa]^{< \kappa} \subseteq I_{\mathcal{F}}$ because $\kappa = \min\{|\mathcal{F}| \colon \mathcal{F} \textrm{ is Kuratowski  partition of }X\}$.
\\
\\
\textbf{2.4.} Let $I$ be an ideal on $\kappa$ and let $S$ be a set with positive measure, i.e. $S \in P(\kappa) \setminus I$. (For our convenience we use $I^+$ instead of $P(\kappa) \setminus I$). 
\\An \textit{$I$-partition} of $S$ is a maximal family $W$ of subsets of $S$ of positive measure such that $A \cap B \in I$ for all distinct $A, B \in W$.

An $I$-partition $W_1$ of $S$ is a \textit{refinement} of an $I$-partition $W_2$ of $S$, ($W_1 \leq~W_2$), iff each $A \in W_1$ is a subset of some $B\in W_2$.

A \textit{functional} $\Phi$ on $S$ is a collection  of functions such that  $ \{dom(f) \colon f \in \Phi\}$ is an $I$-partition of $S$ and $dom(f) \not = dom(g)$, whenever $f \not=g \in \Phi$. This $I$-partition will be denoted by $W_\Phi$.
Elements of functionals will be called \textit{$I$-functions}.
\\
We define $\Phi \leqslant \Psi$ if
\\
(i) each $f \in \Phi \cup \Psi$ is a function into the ordinals;
\\
(ii) $W_\Phi \leq W_\Psi$;
\\
(iii) if $f \in \Phi$ and $g \in \Psi$ are such that $dom(f) \subseteq dom(g)$, then $f(x) < g(x)$ for all $x \in dom(f)$.
\\

If  $I$ is a $\kappa$-complete ideal on $\kappa$ containing singletons, then $I$ is \textit{precipitous} iff whenever $S \in I^+$ and $\{W_n \colon n < \omega\}$ is a sequence of  $I$-partitions of $S$ such that 
$W_0 \geq W_1\geq ... \geq W_n \geq ...$,
then there exists a sequence of sets
$X_0 \supseteq X_1\supseteq ... \supseteq X_n \supseteq ...$
such that $X_n \in W_n$ for each $n\in \omega$ and $\bigcap_{n=0}^{\infty} X_n \not = \emptyset$, (see also \cite[p. 424-425]{TJ}).
\\

The ideal $I_{\mathcal{F}}$ is \textit{an everywhere precipitous ideal} if $I_{\mathcal{F}\cap U}$ is precipitous for each non-empty open set $U \subseteq X$. 
\\

We will need the following characterization of precipitous ideals, (see \cite[Lemma 22.19, p. 424-25]{TJ}).
\\
\\
\textbf{Fact 1 (\cite{TJ})} The following are equivalent
	
	(i) $I$ is precipitous;
	
	(ii) For no $S$ of a positive measure is there a sequence of functionals on $S$ such that $\Phi_0>\Phi_1>...>\Phi_n>...\ .$
\\
\\
\textbf{2.5.} 
Let $FN(\kappa) = \{f \in {^{X}}\kappa \colon \exists_{\mathcal{U}_f} \textrm{ family of open disjoint sets, } \bigcup \mathcal{U}_f \textrm{ is dense in } \\X \textrm{ and } \forall_{F_\alpha \in \mathcal{F}} \forall_{U \in \mathcal{U}_f}\  f \textrm{ is constant on } F_\alpha \cap U\}$.
\\
\\
\textbf{Fact 2 (\cite{FK})}  
If $f, g \in FN(\kappa)$, then 
\\(a) $\{x \colon f(x) < g(x)\}$ has the Baire property,
\\(b) $\{x \colon f(x) = g(x)\}$ has the Baire property.
\\\\
\textbf{2.6.}
Let $\kappa$ be, as previously, a regular uncountable cardinal and let $I$ be a non-principal $\kappa$-complete ideal on $\kappa$. 

Consider an infinite game $\mathcal{G}(I)$ played by two players $\alpha$ and $\beta$ as follows: $\alpha$ moves first by choosing a set $A_0 \in I^+$. Then $\beta$ chooses a set $B_0 \subseteq A_0$ such that $B_0 \in I^+$. Then $\alpha$ chooses $A_1 \subseteq B_0$ such that $A_1 \in I^+$ and so on. Thus, players produce a sequnece  of sets
$$A_0 \supseteq B_0\supseteq A_1 \supseteq B_1 \supseteq ...$$
where $A_i, B_i \in I^+$, $i \in \omega$.
Player $\alpha$ wins iff $\bigcap_{n \in \omega}A_n = \emptyset$. Then we say that the game $\mathcal{G}(I)$ is \textit{$\alpha$-favorable}. 
\\
\\
\textbf{Fact 3 (\cite{JG})} Let $\kappa$ be a regular uncountable cardinal and let $I$ be a non-principal $\kappa$-complete ideal on $\kappa$. Then $I$ is a precipitous ideal iff $\mathcal{G}(I)$ is not $\alpha$-favorable.
\\

Consider an infnite game $\mathcal{G}_1(I)$ played by two players $\alpha$ and $\beta$. $\alpha$ starts the game by choosing an $I$-function $h_0$. Then $\beta$ answers by choosing a  set $B_0 \subseteq dom(h_0)$ such that $B_0 \in I^+$. Then $\alpha$ chooses an $I$-function $h_1$ such that $dom(h_1) \subseteq B_0$ and $h_{1}(\xi) < h_0(\xi)$ for all $\xi \in dom(h_1)$. Players continue the game as is described above producing a sequence
$$(h_0, B_0, h_1, B_1,...)$$ of $I$-functions $h_0, h_1, ...$ such that $dom(h_{n+1}) \subseteq dom(h_n)$  and $h_{n+1}(\xi) < h_n(\xi)$ for all $\xi \in dom(h_{n+1})$ and $B_n \in I^+, (n \in \omega)$. The game  $\mathcal{G}_1(I)$ is $\alpha$-favorable if $\alpha$ can continue the game infinitely. Otherwise $\beta$ wins. 
\\
\\
\textbf{Fact 4 (\cite{JG})} $\mathcal{G}(I)$ is $\alpha$-favorable iff   $\mathcal{G}_1(I)$ is $\alpha$-favorable.
\\
\\
\textbf{Fact 5 (\cite{JG})} Let $\kappa$ be a regular uncountable cardinal and let $I$ be a non-principal $\kappa$-complete ideal on $\kappa$. Then $I$ is a precipitous ideal iff $\mathcal{G}_1(I)$ is not $\alpha$-favorable.

\section{Main result}

\begin{theorem}
	Let $X$ be a Baire metric space with Kuratowski partition $\mathcal{F}$ of cardinality $\kappa$, where $\kappa = min\{|\mathcal{K}|\colon \mathcal{K} \textrm{ is Kuratowski partition of } X\}$.   Then there exists an open set $U \subset X$ such that the $K$-ideal $I_{\mathcal{F}\cap U}$ on $\kappa$ associated with $\mathcal{F}\cap U$ is precipitous.
\end{theorem}

\begin{proof}
	Let $\mathcal{F} = \{F_\alpha \colon \alpha < \kappa\}$, (as was fixed in Section 2.2).
	We will show that there exists an open set $U \subset X$ such that 
	$$I_{\mathcal{F}\cap U} = \{A \subset \kappa \colon \bigcup_{\alpha \in A} F_\alpha\cap U \textrm{ is meager}\}$$ is precipitous.
	
	Suppose that  for any  open $U \subset X$ the ideal $I_{\mathcal{F}\cap U}$ is not precipitous. 
	\\
	Fix a family $\mathcal{U}$ of open and disjoint subsets of $X$ such that $\bigcup \mathcal{U}$ is dense in $X$ and fix $U \in \mathcal{U}$.
	Then by Fact 1 there exists a sequence of functionals 
$\Phi^U_{0} > \Phi^U_{1}> ... $ on some set $S^{U} \in  I^{+}_{\mathcal{F}\cap U}$.
	Let $W^U_k= W_{\Phi^U_k}$ be an $I_{\mathcal{F}\cap U}$-partition (defined in Section 2.4).
	 Let $X^U_k \subset S^U$ be such that $X^U_k \in W^U_k$ for any $k \in \omega$.
	Since $I_{\mathcal{F}\cap U}$ is not precipitous, $\bigcap_{k\in \omega} X^U_{k} = \emptyset$ for any $X^U_{k} \in W^U_{k}$, $ k \in \omega.$
	
Each  $X^U_{k}$ is the domain of some $I_{\mathcal{F}\cap U}$-function $h^U_{k} \in \Phi^U_{k}$ and if $X^U_k \supseteq X^U_{k+1}$, then
$h^{U}_{k} (\beta) >h^{U}_{k+1}(\beta)$ for all $\beta \in X^{U}_{k+1}$, (see Section 2.4.) 

Now, for any $\beta \in X^U_k$ and any $h^U_k \in \Phi^U_{k}$ define a function $f^U_{k, \beta} \in {^X}\kappa$ such that
\\
1) $dom(f^U_{k, \beta})  = F_\beta \cap U$,
\\
2) $f^U_{k, \beta}(x) = h^U_k(\beta)$ for any $x \in F_\beta \cap U$.
\\
Then, by properties of functions $h^U_k, k \in \omega$  we have that 
$$f^U_{k, \beta}(x) > f^U_{k+1, \beta}(x) \textrm{ for any } x \in F_\beta \cap U.$$
\\
Now, for any $k \in \omega$ consider a function
$$f_k = \bigcup_{U \in \mathcal{U}} \bigcup_{\beta < \kappa} f^U_{k, \beta}.$$ Then $f_k\in FN(\kappa)$ for any $k \in \omega$.
By Fact 2, for any $k \in \omega$ the set
$$V_k = \{x\colon f_k(x) > f_{k+1}(x)\}$$
has the Baire property. Then for any $k \in \omega$ the set $X \setminus V_k$ has also the Baire property and moreover is meager in $X$.
\\Indeed. Suppose that there is $k_0 \in \omega$ for which   $X \setminus V_{k_0} = M \triangle W$ for some  meager $M$ and open $W$ and such that  $(X\setminus V_{k_0}) \cap W$ is nonempty. 
Let $x' \in(X\setminus V_{k_0}) \cap W$. Then $f_{k_0}(x') \leqslant f_{k_0+1}(x')$. But $f_k =  \bigcup_{U \in \mathcal{U}} \bigcup_{\beta < \kappa} f^U_{k, \beta}$ and by 2) in the definition of $f^U_{k, \beta}$ we have that $h^U_{k_0}(\beta) \leqslant h^U_{k_0 +1}(\beta)$. A contradiction to the properties of ${I_{\mathcal{F}\cap U}}$-functions $h^U_k$.
Thus, $V_k$ is co-meager for any $k \in \omega$.

By the Baire Category Theorem, (see e.g. \cite[p. 197-198, 277]{RE}), there exists $x_0 \in \bigcap_{k \in \omega}V_k$. Then $$f_0(x_0) > f_1(x_0) > f_2(x_0) > ...$$
what is impossible since $f_k(x_0)$ are ordinals.
\end{proof}

\begin{theorem}
Let $X$ be a Baire metric space with Kuratowski partition $\mathcal{F}$ of cardinality $\kappa$, where $\kappa = min\{|\mathcal{K}|\colon \mathcal{K} \textrm{ is Kuratowski partition of } X\}$, and let $I_\mathcal{F}$ be a $K$-ideal on $\kappa$ associated with $\mathcal{F}$. Then the game $\mathcal{G}_1(I_{\mathcal{F}\cap U})$ is not $\alpha$-favorable for some open set $U \subset X$.
\end{theorem}

\begin{proof}
	Let $\mathcal{F} = \{F_\alpha \colon \alpha < \kappa\}$, (as was fixed in Section 2.2). 	We will show that there exists an open set $U \subset X$ such that the game $\mathcal{G}_1(I_{\mathcal{F}\cap U})$ is not $\alpha$-favorable, where
	$$I_{\mathcal{F}\cap U} = \{A \subset \kappa \colon \bigcup_{\alpha \in A} F_\alpha\cap U \textrm{ is meager, } F_\alpha \in \mathcal{F}\}.$$ 
	\indent
	Suppose that for any open set $U \subset X$ the game $\mathcal{G}_1(I_{\mathcal{F}\cap U})$ is $\alpha$-favorable.
	
	Fix a family $\mathcal{U}$ of open and disjoint subsets of $X$ such that $\bigcup \mathcal{U}$ is dense in $X$ and fix $U \in \mathcal{U}$. 
	$\alpha$ starts the game $\mathcal{G}_1(I_{\mathcal{F}\cap U})$ by choosing $I_{\mathcal{F}\cap U}$-function $h^U_0$. Then $\beta$ chooses $B^U_0 \subseteq dom (h^U_0)$ such that $B^U_0 \in I^+_{\mathcal{F}\cap U}$. Then $\alpha$ chooses $I_{\mathcal{F}\cap U}$-function $h^U_1$ such that $dom(h^U_0) \subseteq B^U_0$ and $h^U_{1}(\beta) < h_0(\beta)$ for all $\beta \in dom(h^U_1)$. Then $\beta$ choose  $B^U_1 \subseteq dom(h^U_1)$ such that $B^U_1 \in  I^+_{\mathcal{F}\cap U}$ and so on.
	Let 	$$(h^{U}_{0}, B^{U}_{0},h^U_1, B^U_1..., h^U_n),$$
	be a finite sequence obtained in this game, $(n \in \omega)$. 
	
	Let $T^U$ denotes the set of all such finite sequences. Note that $T^U$ ordered by extension of sequences is a tree, but our consideration below we will provided for a fixed path of $T^U$. 
	
	Since $I_{\mathcal{F}\cap U}$ is not precipitous, $\alpha$ has a winning strategy. Thus, we have
	$h^U_{k+1}(\beta) < h^U_k(\beta)$ for all $\beta \in dom(h^U_{k+1})$ and $\alpha$ has a legal move for any $k \in \omega$. 
	
	Denote $X^U_k = dom(h^U_k)$, for $U \subset X$ and $k \in \omega$.
	Then by our construction 
	$$B^U_k : = X^U_{k+1} = \{\beta \colon h^U_k(\beta) > h^U_{k+1}(\beta)\}.$$
	\indent
Now, for each $\beta \in B^U_k$ and $h^U_k \in \Phi^U_{k}$ define exactly one function $f^U_{k, \beta} \in {^X}\kappa$ such that
\\
1) $dom(f^U_{k, \beta})  = F_\beta \cap U$,
\\
2) $f^U_{k, \beta}(x) = h^U_k(\beta)$ for any $x \in F_\beta \cap U$.
\\
Then by properties of functions $h^U_k, k \in \omega$  we have that 
$$f^U_{k, \beta}(x) > f^U_{k+1, \beta}(x) \textrm{ for any } x \in F_\beta \cap U.$$

Since our considerations have been provided for any $U \in \mathcal{U}$, we can take a function
$f_k = \bigcup_{U \in \mathcal{U}} \bigcup_{\beta < \kappa} f^U_{k, \beta},$  for any $k \in \omega$. Then $f_k\in FN(\kappa)$, $k \in \omega$.
By Fact 2, for any $k \in \omega$ the set
$V_k = \{x\colon f_k(x) > f_{k+1}(x)\}$
has the Baire property and, moreover, is co-meager, (see the adequate part of the proof of Theorem 1). 

By the Baire Category Theorem, (see e.g. \cite[p. 197-198, 277]{RE}), there exists $x_0 \in \bigcap_{k \in \omega}V_k$. Then $$f_0(x_0) > f_1(x_0) > f_2(x_0) > ...$$
what is impossible since $f_k(x)$ are ordinals. Thus, the game $\mathcal{G}_1(I_{\mathcal{F}\cap U})$ is not $\alpha$-favorable for some open  set $U \subset X$.
	 \end{proof}
\section{Open problem}
 
 In \cite{JG} one can find other "types" of precipitous ideals: weakly-precipitous, (equivalent to some game introduced by S. Shelah) and pseudo-precipitous. Since both mentioned notions are consistent to something stronger than "measurable"  namely $"\kappa^+$-saturated", the natural question  is arisen, whether one can consider theorem adequate for Theorem 1 (and Theorem 2) but for these  notions. 
 But such theorems may occur false, because in  game 
 characterisation of weakly- and pseudo-precipitousness
 the assumption that an ideal must be normal is important. Now we cannot explicity state whether the $K$-ideal from Theorem 1  is normal. Furthemore, there are models of ZFC in which there are precipitous ideals which are not normal precipitous ideals, (see e.g. \cite{G1}). 
 Summarizing, the followng question has been arisen.
 \\
 \\
 \textbf{Open Problem.} Considering assumptions given in Theorem 2, does there exist an open set $U \subseteq X$ for which $I_{\mathcal{F}\cap U}$ is a normal precipitous ideal? If yes, whether $I_{\mathcal{F}\cap U}$ can be pseudo-precipitous (defined in \cite{JG})?

{\sc Ryszard Frankiewicz}
\\
Silesian Univercity of Technology, Gliwice, Poland.
\\
{\sl e-mail: ryszard.frankiewicz@polsl.pl}
\\

{\sc Joanna Jureczko}
\\
Wroc\l{}aw University of Science and Technology, Wroc\l{}aw, Poland
\\
{\sl e-mail: joanna.jureczko@pwr.edu.pl}


\begin {thebibliography}{123456}
\thispagestyle{empty}

\bibitem{LB} L. Bukovsk\'y,    
Any partition into Lebesgue measure zero sets produces a non-measurable set, 
Bull. Acad. Polon. Sci. S\'er. Sci. Math. 27(6) (1979) 431 - 435.

\bibitem{CK} C. C. Chang, H. J. Keisler, Model Theory, North Holland, 1978.

\bibitem{C} G. Choquet,
Lectures on analysis. Vol. I: Integration and topological vector spaces, New York-Amsterdam 1969.

\bibitem{EFK}  A. Emeryk, R. Frankiewicz and W. Kulpa,
On functions having the Baire property,
Bull. Ac. Pol.: Math.  27 (1979) 489--491.

\bibitem{RE}  R. Engelking,
General Topology,
Heldermann Verlag, Berlin, 1989.

\bibitem{FK}  R. Frankiewicz and K. Kunen,
Solutions of Kuratowski's problem on functions having the Baire property, I,
Fund. Math. 128(3) (1987) 171--180.

\bibitem{G1} M. Gitik,
A model with a precipitous ideal, but no normal
precipitous ideal, 
J. Math. Log. 13 (2013), no. 1, 1250008, 22 pp. 

\bibitem {TJ}  T. Jech, 
Set Theory,
The third millennium edition, revised and expanded. Springer Monographs in Mathematics. Springer-Verlag, Berlin, 2003.

\bibitem{JG}{\sc T. Jech},
Some properties of $\kappa$-complete ideals defined in terms of infinite games,
Ann. Pure and Appl. Logic 26 (1984) 31--45.

\bibitem{JJ} J. Jureczko,
The new operations on complete ideals,  Open Math. 17 (2019), no. 1, 415–422.

\bibitem{KK}  K. Kuratowski,
Quelques problem\'es concernant les espaces m\'etriques nonseparables,
Fund. Math. 25 (1935) 534--545.

\bibitem{KK1} K. Kuratowski,
Topology, vol. 1,
Academic Press, New York and London, 1966.

\end {thebibliography}
\end{document}